\documentclass[11pt]{article}%
\usepackage{algorithm}
\usepackage{algpseudocode}
\usepackage{amssymb,amsfonts,amsmath,amsthm}
\usepackage{pdflscape}
\usepackage{rotating} 
\usepackage{longtable}
\usepackage{bbm}
\usepackage{graphicx}%
\usepackage[caption=false]{subfig}
\usepackage{float}
\usepackage{tikz}
\usetikzlibrary{shapes,arrows,arrows.meta,positioning,calc}
\usepackage{pgfplots}
\usepackage{verbatim}
\usepgfplotslibrary{statistics}
\pgfplotsset{compat=1.8}
\usepackage[numbered]{matlab-prettifier}
\usepackage{url}
\usepackage{hyperref}

\newcommand{\sara}[1]{\textcolor{red}{#1}}
\newcommand{\correz}[1]{\textcolor{red}{#1}}

\usepackage{tkz-graph}
\newtheorem{theorem}{Theorem}

\newtheorem{conjecture}[theorem]{Conjecture}
\newtheorem{corollary}[theorem]{Corollary}

\newtheorem{lemma}[theorem]{Lemma}

\newcommand{\df}[2]{\frac{\textnormal{d}  #1}{\textnormal{d} #2}}

\newcommand{\fbr}{f_{\textnormal{br}}} 
\newcommand{\fcr}{f_{\textnormal{cr}}} 
\newcommand{\fbrhet}{f_{\textnormal{br,ij}}} 
\newcommand{\fcrhet}{f_{\textnormal{cr,ij}}} 
\newcommand{\fbrsame}{f_{\textnormal{br,ii}}} 
\newcommand{\fcrsame}{f_{\textnormal{cr,ii}}} 
\newcommand{\yinfty}{\bar{y}_\infty} 

\newcommand{\BA}{Barab{\'a}si-Albert}

\begin{document}
	
\title{A minimal model for multigroup adaptive SIS epidemics}
	
\author{Massimo A. Achterberg$^{1,}$\thanks{M.A.Achterberg@tudelft.nl}\;,\; Mattia Sensi$^{2,3,}$\thanks{mattia.sensi@polito.it}\; and Sara Sottile$^{4,}$\thanks{sara.sottile4@unibo.it}}
\date{\small $^1$Faculty of Electrical Engineering, Mathematics and Computer Science \\ Delft University of Technology, P.O. Box 5031, 2600 GA Delft, The Netherlands \\
$^2$MathNeuro Team, Inria at Universit\'e C\^ote d'Azur, 2004 Rte des Lucioles, 06410 Biot, France\\
$^3$Department of Mathematical Sciences ``G. L. Lagrange'', Politecnico di Torino,\\ Corso Duca degli Abruzzi 24, 10129 Torino Italy\\
$^4$Dept. of Medical and Surgical Sciences, University of Bologna, Via Massarenti 9, Bologna Italy}

\maketitle

\begin{abstract}
We propose a generalization of the adaptive N-Intertwined Mean-Field Approximation (aNIMFA) model studied in \emph{Achterberg and Sensi} \cite{achterbergsensi2022adaptive} to a heterogeneous network of communities. In particular, the multigroup aNIMFA model describes the impact of both local and global disease awareness on the spread of a disease in a network. We obtain results on existence and stability of the equilibria of the system, in terms of the basic reproduction number~$R_0$. \correz{Assuming individuals have no reason to decrease their contacts in the absence of disease}, we show that the basic reproduction number~$R_0$ is equivalent to the basic reproduction number of the NIMFA model on static networks.
Based on numerical simulations, we demonstrate that with just two communities periodic behaviour can occur, which contrasts the case with only a single community, in which periodicity was ruled out analytically.
We also find that breaking connections between communities is more fruitful compared to breaking connections within communities to reduce the disease outbreak on dense networks, but both strategies are viable to networks with fewer links. Finally, we emphasize that our method of modelling adaptivity is not limited to SIS models, but has huge potential to be applied in other compartmental models in epidemiology.
\end{abstract}

\section{Introduction}\label{sec_introduction}
\correz{During epidemic outbreaks without access to medicines, vaccines and other pharmaceutical prevention measures, the only possibility to reduce the spread of the outbreak is by contact avoidance. If the choice for contact avoidance is based on the prevalence of the disease in the population, we call it \emph{adaptivity}.} Various approaches have been suggested to model the response of individuals to an epidemic. A common approach is to apply an SIR-like (Susceptible -- Infected -- Recovered) compartmental model by including an Aware compartment \cite{juher2020robustness,juher2022saddle,just2018oscillations,sahneh2012optimal,sahneh_2019_awareness_SIS}. Aware individuals are conscious of the ongoing disease and are therefore more cautious towards contact with other individuals. Thus, aware individuals are less likely to be infected. Moreover, some authors included the possibility of corruption of information spreading throughout a population \cite{agliari2006efficiency,funk2009spread,ASISpairapprox}, which has been very problematic throughout the \mbox{COVID-19} pandemic in particular \cite{kouzy2020covid}.

Other authors suggested local rules to model personal risk mitigation using link rewiring schemes \cite{GrossEtAl2006} or link-breaking rules \cite{RandomLinkDeActivate,G-ASIS}. These models primarily consider personal decisions of nodes to apply mitigation strategies, who generally base their decision on the presence of the disease in the local neighbourhood around that node. On the contrary, the overall presence of the disease in the population is often neglected. A recent paper, conversely, proposed an interesting game-theoretical approach to the coupling of infectious disease spreading and global awareness \cite{cascante2022adaptive}.

In reality, both personal risk perception and global knowledge of the disease persistence in the population play a role in the mitigation strategy. Here, we propose a multigroup adaptive SIS (Susceptible -- Infected -- Susceptible) model, based on the simple model for adaptive SIS epidemics \cite{achterbergsensi2022adaptive}. The multigroup approach is convenient to describe heterogeneity amongst agents using a system of ODEs, by considering the evolution of a disease in a network of communities \cite{LajmanovichYorke1976,huang1992stability,sun2011global,muroya2013global,muroya2014further,mohapatra2015compartmental,fan2018global,adimy2023multigroup,ottaviano2023global}. The disease evolves both within each community, in which homogeneous mixing is assumed, and between communities, depending on the strength of the links between those communities. The size of a community may vary between households to entire countries. Besides the usual split-up of the population in geographically separated groups, the communities mentioned above could also represent age-groups in the population and they are widely used for applications with real data to simulate the incidence of a disease stratified by age groups
\cite{boccalini2021health,calabro2022health,fochesato2022economic}. 

Even though the disease dynamics of the aNIMFA model is extremely simplified (only an SIS model is utilised), we expect the dynamics to be very rich and should be able to accurately capture many aspects of group-level risk mitigation during epidemic outbreaks. In particular, we show that periodic solutions may occur in an asymmetric network with just two communities, which contrasts the case with a single community, in which it was proven that limit cycles cannot occur \cite{achterbergsensi2022adaptive}.

The paper is structured as follows. In Section \ref{sec:model}, we introduce our multigroup generalization of the aNIMFA model. In Section \ref{sec:results}, we compute the basic reproduction number and prove the existence and stability of the disease-free equilibrium. Then, we prove the existence of at least one endemic equilibrium when the basic reproduction number is (slightly) larger than one. Section \ref{sec:simul} showcases various numerical case studies of the aNIMFA model, including the emergence of periodic behaviour in just two communities. We conclude with Section \ref{sec:concl}.

\section{The multigroup aNIMFA model}\label{sec:model}
Achterberg and Sensi \cite{achterbergsensi2022adaptive} introduced a simple model for adaptive SIS epidemics in a well-mixed population. The model is given by
\begin{subequations}\label{eq_animfa_1node}
	\begin{align}
		\df{y}{t} &= - \delta y + \beta y (1-y) z,\label{eq_diff_y1} \\
		\df{z}{t} &= - \zeta z \fbr(y) + \xi (1-z) \fcr(y),\label{eq_diff_z1} \\
		& \qquad \qquad \text{feasible region } 0 \leq y \leq 1, 0 \leq z \leq 1 \nonumber
	\end{align}
\end{subequations}
where $y,z$ represent the fraction of infected nodes (known as the \emph{prevalence}) and the fraction of existing connections in the population, respectively. The fraction of infected nodes in (\ref{eq_diff_y1}) decreases based on recovery (first term with rate $\delta$) and increases based on infections (second term with rate~$\beta$). The link density in (\ref{eq_diff_z1}) decreases (first term) based on the link-breaking process $\fbr$ with rate $\zeta$ and increases (second term) with the link-creation process $\fcr(y)$ with rate $\xi$. The functional responses $\fbr$ and $\fcr$ capture human behaviour and describe the people's response to the ongoing epidemic. In particular, the functional responses $\fbr$ and $\fcr$ depend directly on the prevalence~$y$. It is assumed that $\fbr$ and $\fcr$ are non-negative functions, i.e. $\fbr(y),\fcr(y) \geq 0$ for all $0 \leq y \leq 1$.

Here, we propose a generalisation of (\ref{eq_animfa_1node}) to $n$ heterogeneous groups:
\begin{subequations}\label{eq_animfa}
	\begin{align}
		\df{y_i}{t} &= -\delta_i y_i + (1-y_i) \sum_{j=1}^n \beta_{ij} y_j  z_{ij}, \\
		\df{z_{ij}}{t} &= - \zeta_{ij} z_{ij} \fbrhet(y_i, y_j, \bar{y}) + \xi_{ij} (1-z_{ij}) \fcrhet(y_i, y_j, \bar{y}), \\
		\bar{y} &= \frac{1}{n} \sum_{j=1}^n y_j,\\
		& \qquad \qquad \text{feasible region } 0 \leq y_i \leq 1, 0 \leq z_{ij} \leq 1 \text{ for } 1 \leq i,j \leq n \nonumber
	\end{align}
\end{subequations}
where $y_i$ is the local prevalence in group $i$, $z_{ij}$ is the link density between group \correz{$j$ and $i$ (which is not necessarily equal to $z_{ji}$)}, $\delta_i$ is the curing rate of group $i$, $\beta_{ij}$ is the infection rate from group $j$ to group $i$, $\zeta_{ij}$ is the link-breaking rate between group \correz{$j$ and $i$}, $\xi_{ij}$ is the link-creation rate between group \correz{$j$ and $i$} and $\bar{y}$ is the global prevalence (or simply the prevalence) of the infectious disease in the whole population, i.e.\ the average prevalence over all communities. We are implicitly assuming here that all communities have the same size. In general, the functional responses $\fbrhet$ and $\fcrhet$ can be different between each \correz{couple of communities}. In particular, when $i=j$, the functional responses $\fbrsame$ and $\fcrsame$ describe the different responses of each community $i$ to the internal spread of the infectious disease. \correz{We represent the exemplifying flow within and between two connected communities in Figure \ref{fig:flow}.}

	\begin{figure}[h!]
		\centering
		\begin{tikzpicture}
			\node[draw,circle,thick,minimum size=1.4cm] (Si) at (-2,2) {$1-y_i$};
			\node[draw,circle,thick,minimum size=1.4cm] (Ii) at (2,2) {$y_i$};
			\draw[-Stealth,thick] (Si)--(Ii) node[below, midway]{$\beta_{ii}z_{ii}y_i$};
			\draw[-Stealth,thick] (Ii) to [bend right] node[above, midway]{$\delta_i$} (Si) ;
			\node[draw,circle,thick,minimum size=1.4cm] (Ij) at (-2,-1) {$y_j$};
			\node[draw,circle,thick,minimum size=1.4cm] (Sj) at (2,-1) {$1-y_j$};
			\draw[-Stealth,thick] (Sj)--(Ij) node[above, midway]{$\beta_{jj}z_{jj}y_j$};
			\draw[-Stealth,thick] (Ij) to [bend right] node[below, midway]{$\delta_j$} (Sj) ;
			\draw[-Stealth,dashed] (Ij)--(Si) node[left, midway]{$\beta_{ij}z_{ij}y_j$};
			\draw[-Stealth,dashed] (Ii)--(Sj) node[right, midway]{$\beta_{ji}z_{ji}y_i$};
			\node at (-2.8,2.8) {\Large$S_i$};
			\node at (2.8,2.8) {\Large$I_i$};
			\node at (-2.8,-1.8) {\Large$I_j$};
			\node at (2.8,-1.8) {\Large$S_j$};
		\end{tikzpicture}
		\caption{\correz{Interaction within and between two connected communities $i$ and $j$. Solid lines: change of state of individuals within one community; dashed lines: inter-community infections. Notice that the position of the Susceptible and Infected compartments are switched from the top and the bottom row.} 
			\label{fig:flow}}
	\end{figure}
\correz{We remark that we only need to keep track of the fraction of Infected individuals in each community, since Susceptible individuals are exactly $1-y_i$ at all times, as we are considering an SIS model with no demography spreading in a network of communities with normalized populations.}

The $n$ communities can be represented as $n$ nodes in a graph $G$, where each node represents one community. The links in the graph $G$ constitute the connections \correz{$\beta_{ji}$ from community~$i$ to community~$j$}. We assume the graph $G$ is \correz{strongly connected} (i.e.\ the adjacency matrix $B=(\beta_{ij})$ is irreducible). \correz{Although the infection rate $\beta_{ij}$ may be positive from community $j$ to $i$, the adaptive weight~$z_{ij}$ may become (temporarily) zero due to the link-adaptivity process. On the other hand, a non-existent link $\beta_{ij}=0$ in the adjacency matrix~$B$ can never yield a time-varying (positive) link weight $z_{ij} > 0$, because there is simply no interaction from community $j$ to $i$. Consequently, one may chose to ignore entirely the ODEs corresponding to $z_{ij}$ for which $\beta_{ij}=0$, or set the corresponding adaptivity functions and parameters to zero.} 

From a biological point of view, a natural assumption would be to consider all $\delta_i$ equal, since they represent the inverse of the recovery rate of the same disease. However, each region may have a different healthcare quality, leading to potentially slightly different recovery rates $\delta_i$. Therefore, we keep the model \eqref{eq_animfa} as general as possible.

Equation (\ref{eq_animfa}) is coined the multigroup adaptive N-Intertwined Mean-Field Approximation (multigroup aNIMFA), named after the NIMFA model \cite{IntroNIMFA}. On a technical note, Eq.\ (\ref{eq_animfa}) contains \correz{at most} $n+n^2$ equations \correz{(one for each community, and at most $n^2$ for the connections among them)}, not $n$, but we will adhere to the definition (\ref{eq_animfa}) due its analogy to the static NIMFA model~\cite{IntroNIMFA}. We assume the parameters $\delta_i$, $\beta_{ij}$, $\zeta_{ij}$, $\xi_{ij}$ can be different for each node and link, and are assumed to be non-negative. The standard NIMFA model with link weights $z_{ij}=z_{ij}(0)$ is recovered if $\zeta_{ij}=\xi_{ij}=0$ for all $i$ and $j$. Furthermore, $\fbrhet$ and $\fcrhet$ are assumed to be non-negative for all $0 \leq y \leq 1$ and all $i$ and $j$. We emphasise here that the link density in the aNIMFA model is not necessarily symmetric, i.e.\ $z_{ij} \neq z_{ji}$. This allows us to capture unilateral decisions to limit or encourage movement between community $i$ and community $j$, but not vice versa. For example, we can think of some country A which lifts restrictions on incoming flights from another country B, even though country B still discourages its inhabitants to travel.

In an epidemic context, it is reasonable to assume that the breaking (respectively, creation) of contacts is non-decreasing (non-increasing) with respect to the current prevalence. Indeed, during epidemic peaks, people are more likely to isolate, and during periods of low prevalence, people are likely to enhance their social activities. Mathematically, this translates in the functions $\fbrhet$ and $\fcrhet$ being non-decreasing and non-increasing in all their arguments, respectively. For the remainder of this work, we focus on the epidemic case and therefore assume non-decreasing $\fbrhet$ and non-increasing $\fcrhet$.

Equations (\ref{eq_animfa}) describe the interplay between \emph{local} and \emph{global} awareness, where the functional response of the individuals to the disease is based on the local prevalences~$y_i$ and $y_j$, and global prevalence~$\bar{y}$, respectively. Similarly, the aNIMFA model distinguishes between the within-link-density~$z_{ii}$ and cross-link-density $z_{ij}$, which allows for different or targeted counter-measures in specific communities.

We remark that system \eqref{eq_animfa} may be derived as a mean field limit of a stochastic model, as was done in \cite{achterberg2020classification,achterberg2022moment}. In particular, it is a generalization of \cite[Eq. (4)]{achterberg2022moment}, which allows us to include non-local information on the prevalence in the adaptivity mechanism.

Prior to investigating specific case studies, we first prove several results for the general case of system~\eqref{eq_animfa}.

\section{Analytical results}\label{sec:results}
In this section, we provide numerous analytical results concerning system \eqref{eq_animfa}.

\subsection{Boundedness of the feasible region}\label{sec_bounded}
In the following \correz{Lemma}, we prove that system \eqref{eq_animfa} evolves in the biologically relevant region $[0,1]^{n+n^2}$.
\begin{lemma}\label{lem:bounded}
Consider a solution of system \eqref{eq_animfa} starting at $y_i(0) \in [0,1]$ and $z_{ij}(0) \in [0,1]$ for all $i$ and~$j$. Recall that $\fbrhet(y_i,y_j,\Bar{y}),\fcrhet(y_i,y_j,\Bar{y})\geq 0$ for all $y_i,y_j,\Bar{y}\in [0,1]$ and all $i,j$. Then, $y_i(t),z_{ij}(t) \in [0,1]$ for all $t \geq 0$ and all $i$ and $j$.
\end{lemma}
\begin{proof}
    We calculate
    $$
    \df{y_i}{t}\bigg|_{y_i=0}=\sum_{j\neq i} \beta_{ij} y_j  z_{ij}\geq 0, \quad \df{y_i}{t}\bigg|_{y_i=1}=-\delta_i<0,$$
    $$\df{z_{ij}}{t}\bigg|_{z_{ij}=0}=\xi_{ij} \fcrhet(y_i,y_j,\Bar{y}) \geq 0, \quad \df{z_{ij}}{t}\bigg|_{z_{ij}=1}=-\zeta_{ij}\fbrhet(y_i,y_j,\Bar{y}) \leq 0,
    $$
    which proves the forward invariance in the interval $[0,1]$ of each $y_i$, $i=1,\dots,n$, and $z_{ij}$, $i,j=1,\dots,n$.
\end{proof}

\subsection{Disease-Free Equilibrium}\label{sec_dfe}
The Disease-Free Equilibrium (DFE) of system \eqref{eq_animfa} corresponds to the state in which the prevalence in each community is zero, i.e.\ $y_i=0$ for $i=1,2,\dots,n$. Consequently, the link densities at the DFE are
$$
z^{\text{DFE}}_{ij}=\dfrac{\xi_{ij} \fcrhet(\mathbf{0})}{\zeta_{ij} \fbrhet(\mathbf{0})+\xi_{ij} \fcrhet(\mathbf{0})},
$$
where, for ease of notation, we write $\fbrhet(\mathbf{0}):=\fbrhet(0,0,0)$ and $\fcrhet(\mathbf{0}):=\fcrhet(0,0,0)$.

\subsection{Basic Reproduction Number}\label{sec_R0}
We apply the Next Generation Matrix method \cite{van2008further} to compute the Basic Reproduction Number $R_0$ of system \eqref{eq_animfa}. We evaluate the Jacobian relative to the variables $y_i$ at the DFE, and write it as
$$
J=M-V,
$$
where
$$
M_{ij}=\beta_{ij}z^{\text{DFE}}_{ij}, \quad V=\text{diag}(\delta_1,\dots,\delta_n).
$$
\correz{Notice that the matrix $M$ represents infections, whereas the matrix $V$ represents recovery.} Then, the basic reproduction number $R_0$ follows as the largest eigenvalue $\rho$ of the matrix $F$
\begin{equation}\label{eq_R0}
    R_0 = \rho(F),
\end{equation}
where
\begin{equation}\label{eq_Fmatrix}
F=MV^{-1}=\left(  \dfrac{\beta_{ij}}{\delta_i} z^{\text{DFE}}_{ij}\right)_{1\leq i,j \leq n}.
\end{equation}
Next, we consider the following two special cases.

Assuming complete homogeneity of the parameters, i.e. $\delta_i=\delta, \beta_{ij}=\beta, \zeta_{ij}=\zeta$, $\xi_{ij}=\xi$, $\fcrhet(y)=\fcr(y)$ and $\fbrhet(y)=\fbr(y)$ for all $i,j$, the matrix $F$ becomes a rank 1 matrix, since all its entries are equal. Then, $R_0$ is exactly equal to $n$ times the basic reproduction number for a single community \cite[Sec. 3.4]{achterbergsensi2022adaptive}, namely
\begin{equation*}
    R_{0,\text{hom}} = n \dfrac{\beta}{\delta} \dfrac{\xi \fcr(\mathbf{0})}{\zeta \fbr(\mathbf{0})+\xi \fcr(\mathbf{0})}.
\end{equation*}

\correz{If instead we assume} no link-breaking at zero prevalence, i.e.\ $\fbrhet(\mathbf{0})=0$ for all $i,j$, then the steady-state link density $z_{ij}^{\textnormal{DFE}}=1$ for all $i,j$ and the matrix
\begin{equation*}
    F=MV^{-1}=\left(  \dfrac{\beta_{ij}}{\delta_i}\right)_{1\leq i,j \leq n}
\end{equation*}
is completely independent of $\zeta$, $\xi$, $\fbr$ and $\fcr$. Thus, the basic reproduction number $R_0$ is completely independent of the network dynamics, but not from the network topology, and is equivalent to the NIMFA model on a static topology. This can be explained by the local nature of $R_0$ around the DFE; the definition of this threshold quantity is fundamentally related to the local stability of the DFE. Hence, if in the absence of infection the number of contacts between individuals is not decreased, the topology of the network (and not the network dynamics) completely characterizes the potential spread of the disease.

\subsection{Stability of the DFE}
We now use definition (\ref{eq_R0}) to prove \correz{Theorem}~\ref{thm_conv_DFE}.
\begin{theorem}\label{thm_conv_DFE}
    Assume that $\fbrhet(y_i,y_j,\Bar{y})$ and $\fcrhet(y_i,y_j,\Bar{y})$ are respectively non-decreasing and non-increasing in all their arguments for all $i,j$. Then, the DFE of system \eqref{eq_animfa} is globally stable when $R_0<1$.
\end{theorem}
\begin{proof}
    First, we prove that all the prevalence variables $y_i\rightarrow  0$ as $t \rightarrow +\infty$ when $R_0<1$.

    Recall that, from our assumptions, $\fbrhet$ and $\fcrhet$ are respectively non-decreasing and non-increasing in all their arguments. Then, for each triple $(y_i,y_j,\Bar{y})$, we have $\fbrhet(\mathbf{0})\leq \fbrhet(y_i,y_j,\Bar{y})$ and $\fcrhet(\mathbf{0})\geq \fcrhet(y_i,y_j,\Bar{y})$, respectively. Then,
    $$
    \df{z_{ij}}{t}\leq - \zeta z_{ij} \fbrhet(\mathbf{0}) + \xi (1-z_{ij}) \fcrhet(\mathbf{0}), 
    $$
    from which we can deduce
    $$
    \limsup_{t \to \infty} z_{ij}(t) \leq z_{ij}^{\text{DFE}}.
    $$
    This means that, for each $\varepsilon_1>0$, there exists a time $t_{\varepsilon_1}$ such that, for $t\geq t_{\varepsilon_1}$ and for all $i,j$, we have
    $$
    z_{ij}(t) \leq z_{ij}^{\text{DFE}}+\varepsilon_1.
    $$
    This means that, for $t\geq t_{\varepsilon_1}$, we can bound  the first $n$ ODEs of \eqref{eq_animfa} from above, using moreover the fact that $(1-y_i)\leq 1$, by
    \begin{equation}\label{proof_1}
        \df{y_i}{t}\leq -\delta_i y_i + \sum_{j=1}^N \beta_{ij} y_j (z_{ij}^{\text{DFE}}+\varepsilon_1).  
    \end{equation}
    Consider the following auxiliary system, which is obtained by taking equality in \eqref{proof_1}:
    \begin{equation}\label{proof_2}
        \df{w_i}{t}= -\delta_i w_i + \sum_{j=1}^N \beta_{ij} w_j (z_{ij}^{\text{DFE}}+\varepsilon_1).  
    \end{equation}
    Defining the vector $w:=(w_1,\dots,w_n)$, we can rewrite \eqref{proof_2} as
    \begin{equation}
        \df{w}{t}=(M(\varepsilon_1) -V )w ,  
    \end{equation}
    where $V=\text{diag}(\delta_i)$, as in the definition of $R_0$, and $M(\varepsilon_1)=(\beta_{ij}  (z_{ij}^{\text{DFE}}+\varepsilon_1))_{ij}$. We invoke the following \correz{Lemma}:
    \begin{lemma}[\correz{Lemma 2 in \cite{van2008further}}]\label{eigenlemma}
If $M$ is non-negative and $V$ is a non-singular M-matrix, then $R_0=\rho(MV^{-1})<1$
if and only if all eigenvalues of $(M-V)$ have negative real parts.
\end{lemma}
By picking $\varepsilon_1$ small enough, we can ensure $\rho(M(\varepsilon_1)V^{-1})<1$. Then, from \correz{Lemma} \ref{eigenlemma}, for all $i$ we have $w_i(t) \to 0$ as $t \to \infty$, which implies, by comparison, that $y_i(t) \to 0$ as $t \to \infty$. 

This means that, for each $\varepsilon_2 >0$, there exists a time $t_{\varepsilon_2}$ such that, for $t > t_{\varepsilon_2}$ and for $i=1,\dots,n$, we have 
\begin{equation}\label{proof_3}
y_i \leq \varepsilon_2.
\end{equation}
By substituting \eqref{proof_3} in the ODEs for $z_{ij}$, and recalling our assumption of monotonicity of $\fcrhet$ and $\fbrhet$, we obtain the following inequalities:
\begin{equation}
z_{ij}' \geq - \zeta z_{ij} \fbrhet(\varepsilon_2,\varepsilon_2,\varepsilon_2) + \xi (1-z_{ij}) \fcrhet(\varepsilon_2,\varepsilon_2,\varepsilon_2), 
\end{equation}
from which we can deduce
$$
\liminf_{t \to \infty} z_{ij}(t) \geq \dfrac{\xi_{ij} \fcrhet(\varepsilon_2,\varepsilon_2,\varepsilon_2)}{\zeta_{ij}\fbrhet(\varepsilon_2,\varepsilon_2,\varepsilon_2)+ \xi_{ij} \fcrhet(\varepsilon_2,\varepsilon_2,\varepsilon_2)}.
$$
Taking $\varepsilon_1,\varepsilon_2 \to 0$ concludes the proof.
\end{proof}
We conclude the results on stability of the DFE with the following intuitive \correz{Corollary}:\begin{corollary}\label{cor_DFE_unstable}
    The DFE is locally (hence, globally) unstable when $R_0>1$.
\end{corollary}
The proof of \correz{Corollary} \ref{cor_DFE_unstable} coincides, up to extremely minor adjustements, with the second half of the proof of \correz{Theorem 1 in \cite{van2008further}}, which in turn makes use of a few smaller results presented in the same paper; for the sake of brevity, we do not repeat it here.

\subsection{Endemic Equilibria}
We now prove, under slightly more restrictive assumptions than $R_0>1$, the existence of (at least one) Endemic Equilibrium (EE)\correz{, in which the prevalence in each community is strictly positive, i.e.\ $y_i>0$ for $i=1,2,\dots,n$. Indeed, assuming that at equilibrium one $y_i>0$ necessarily means that all adjacent communities must have strictly positive prevalence in equilibrium. Since the network is assumed to be \correz{strongly connected}, this implies that all $y_i$ are strictly positive at the EE}. We remark that, in all generality, we can not prove stability nor instability of the EE(s). We do not prove uniqueness and, as we shall see in Section \ref{sec:simul}, system \eqref{eq_animfa} admits limit cycles even in a network with only \emph{two} communities.
\begin{theorem}\label{thm:homog}
Assume that the matrix $F$ from Eq.\ \eqref{eq_Fmatrix} is such that the minimum row/column sum is strictly bigger than 1. Then, system \eqref{eq_animfa} admits at least one EE.
\end{theorem}
\begin{proof}
We know that
$$
\min \text{ row/column sum of }F \leq \rho(F) \leq \max \text{ row/column sum of }F,
$$
hence under our assumption, $\rho(F)=R_0>1$.
An equilibrium of system \eqref{eq_animfa} necessarily satisfies 
$$
z^{*}_{ij}=\dfrac{\xi_{ij} \fcrhet(y_i, y_j, \bar{y})}{\zeta_{ij} \fbrhet(y_i, y_j, \bar{y})+\xi_{ij} \fcrhet(y_i, y_j, \bar{y})}.
$$
Substituting this expression for $z^{*}_{ij}$ in the first $n$ equations of \eqref{eq_animfa}, we obtain
\begin{equation}
    \label{eq_for_proof1}
    -\delta_i y_i + (1-y_i) \sum_{j=1}^n \beta_{ij} y_j  \dfrac{\xi_{ij} \fcrhet(y_i, y_j, \bar{y})}{\zeta_{ij} \fbrhet(y_i, y_j, \bar{y})+\xi_{ij} \fcrhet(y_i, y_j, \bar{y})}=0.
\end{equation}
Since 
$$ \df{y_i}{t}\bigg|_{y_i=1}=-\delta_i<0,
    $$
    if we show that for some small $\varepsilon>0$ we have 
$$ \df{y_i}{t}\bigg|_{y_i=\varepsilon}>0,
    $$
for all $i$, we can apply the  Poincaré-Miranda Theorem \cite{kulpa1997poincare,mawhin2019simple} to conclude the existence of \emph{at least} one EE. 

If we study the sign of \eqref{eq_for_proof1} at $y_i =\varepsilon$ for all $i$, we obtain
 \begin{equation}
    \label{eq_for_proof2}
    -\delta_i \varepsilon + (1-\varepsilon) \sum_{j=1}^n \beta_{ij} \varepsilon  \dfrac{\xi_{ij} \fcrhet(\varepsilon, \varepsilon, \varepsilon)}{\zeta_{ij} \fbrhet(\varepsilon, \varepsilon, \varepsilon)+\xi_{ij} \fcrhet(\varepsilon, \varepsilon, \varepsilon)}>0.
\end{equation}   
We can simplify both sides of \correz{Equation} \eqref{eq_for_proof2} by $\varepsilon$, and rearrange it to obtain 
 \begin{equation}
    \label{eq_for_proof3}
\sum_{j=1}^n \dfrac{\beta_{ij}}{\delta_i} \dfrac{\xi_{ij} \fcrhet(\varepsilon, \varepsilon, \varepsilon)}{\zeta_{ij} \fbrhet(\varepsilon, \varepsilon, \varepsilon)+\xi_{ij} \fcrhet(\varepsilon, \varepsilon, \varepsilon)}>\dfrac{1}{1-\varepsilon}.
\end{equation}
If we take $\varepsilon \to 0$, the previous equation coincides with requiring that the sum of row $i$ of the matrix $F$ \eqref{eq_Fmatrix} is (strictly) greater than $1$. This must hold for all $i$, which is a consequence on our assumption on the minimum of such sums. Then, \eqref{eq_for_proof3} holds for $\varepsilon>1$ small enough. This concludes the proof.
\end{proof}
We conjecture the following, based on our extensive numerical simulations:
\begin{conjecture}
  System \eqref{eq_animfa} admits at least one EE when $R_0>1$. 
\end{conjecture}
We remark, moreover, that the number of EE heavily depends on the specific choices for the function $\fcrhet,\fbrhet$. In order to reach stronger conclusions on their cardinality and stability, one needs to either select specific examples or impose additional conditions.

\section{Numerical simulations}\label{sec:simul}
The complexity of the governing \correz{Equations} \eqref{eq_animfa} complicates deriving further results on the dynamics. We therefore resort to numerical simulations of system \eqref{eq_animfa} in specific scenarios. The simulations in this section have been executed in Matlab and are based on a simple Forward Euler scheme with time step $\Delta t = 0.01$ of system~\eqref{eq_animfa}. \correz{In the simulations, we were careful to account for the fact that the adaptivity of the weights $z_{ij}$   may vary the intensity of contacts between two communities connected by the corresponding $\beta_{ij}$, but does not create new connections not encoded in the adjacency matrix $B$. This consideration ensures that the numerical implementation is efficient, as it reduces the number of differential equations that need to be integrated.} \correz{The parameters $\delta_i$, $\beta_{ij}$, $\zeta_{ij}$ and $\xi_{ij}$ are all rates, i.e. $1$ over unit time.}

We present \correz{three} case studies on the multigroup aNIMFA model, each discussing a key aspect of the interplay between infectious disease dynamics and disease awareness.

\correz{The codes used to produce the simulations contained in this section are available in the GitHub repository\footnote{\url{https://github.com/SaraSottile/AdaptativeMeanField}}. The parameters for Case Study 1 are provided in Table \ref{tab:param1}, while the parameters for Case Studies 2 and 3 can be generated from the codes by using the provided random seed. All the Figures are based on a single realization of the code. Refer to Section \ref{sec:codes} for the codes used to produce all the Figures.}

\subsection{Case study 1: Periodic behaviour}\label{C1}
Due to the interplay between disease dynamics and human behaviour, one would expect that periodicity emerges naturally in the multigroup aNIMFA model. For the aNIMFA model in one isolated community, it was recently proven that the dynamical equations do not admit periodic behaviour \cite{achterbergsensi2022adaptive}. For the multigroup aNIMFA model, we demonstrate here that periodicity may occur with just two communities for certain choices of the functional responses $\fbr$ and $\fcr$.

\begin{figure*}[!ht]
    \centering
    \subfloat[\label{fig_ex0a} Local and average prevalence.]{%
        \includegraphics[width=0.49\textwidth]{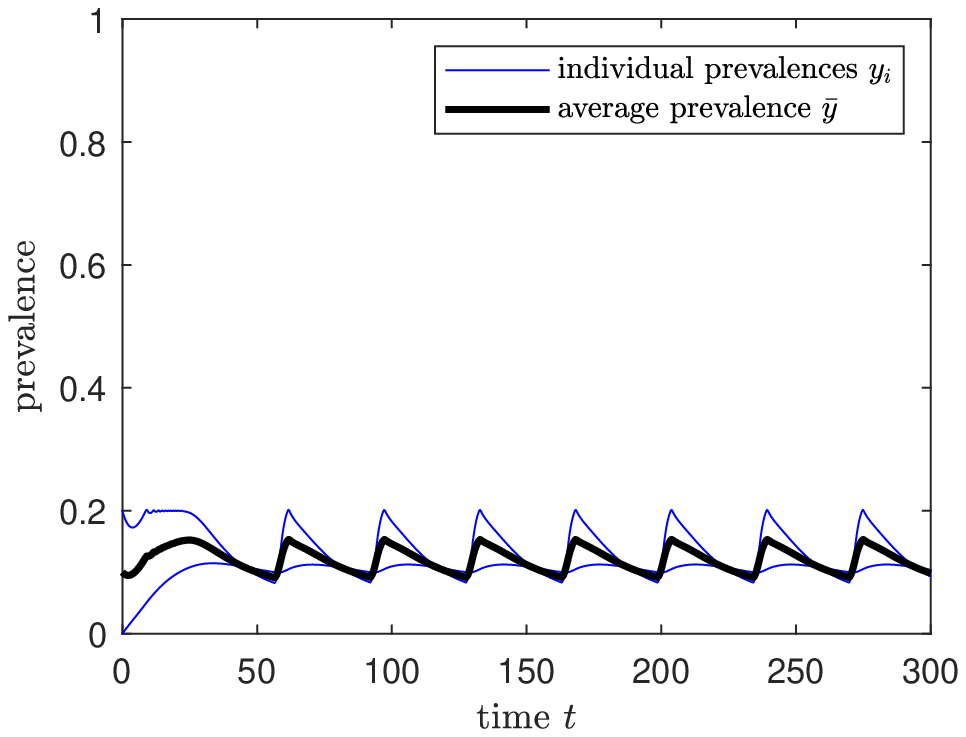}
    }
        \subfloat[\label{fig_ex0b} Link densities.]{%
        \includegraphics[width=0.49\textwidth]{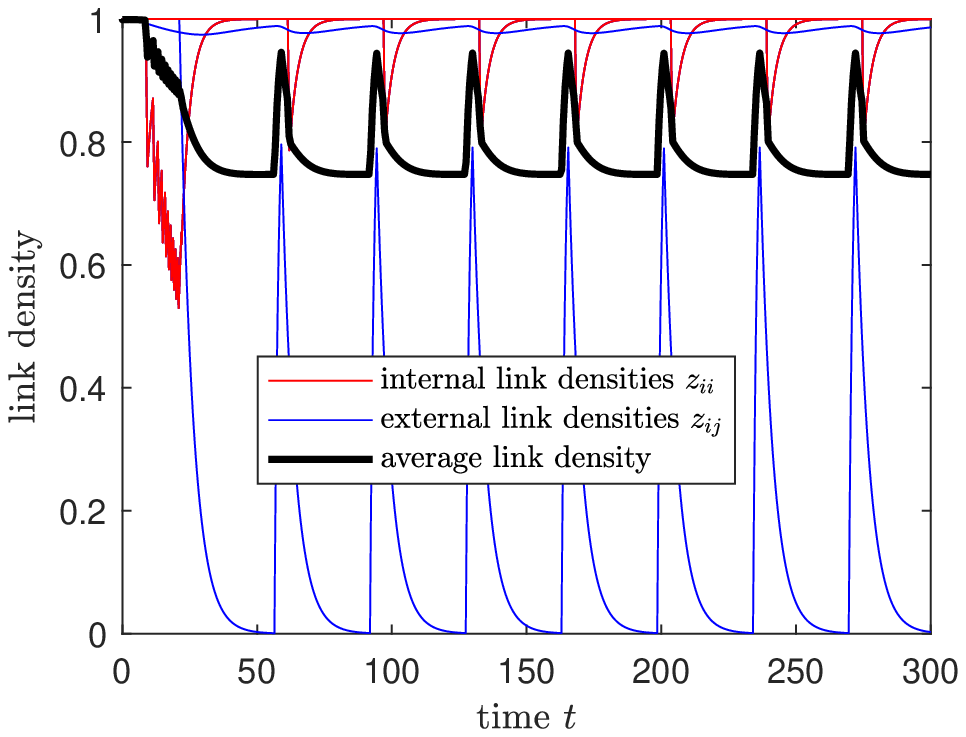}
    }
    \caption{Case study 1: effect of asymmetric response between two communities. Parameters are \correz{listed in Table \ref{tab:param1}}. (a) local and average prevalence; (b) link densities. We observe a quick convergence towards a stable limit cycle, representing endemicity of the disease alternating between high and low prevalence in the population.}
    \label{fig_ex0b_p}
\end{figure*}

\begin{table}[]
    \centering
  \correz{  \begin{tabular}{|c|c|}
 \hline   Parameter & Value\\\hline
       $n$  & $2$ \\\hline
        $\beta_{ij}$ & $\left(\begin{array}{cc}
           0.52  & 0.66 \\
       0.03      & 0.52
        \end{array}\right)$ \\\hline
        $\delta_1=\delta_2$ & $
           0.50$\\\hline
           $\zeta_{ij} $&$\left(\begin{array}{cc}
          0.42&0.20\\
0.33 &	0.62
        \end{array}\right)$\\\hline
        $\xi_{ij}$ &$\left(\begin{array}{cc}
           0.30  & 0.62 \\
       0.27 &0.53
        \end{array}\right)$\\
        \hline
        $z_{ij}(0)$ &  $\left(\begin{array}{cc}
           1  & 1 \\
       1 & 1 
        \end{array}\right)$\\\hline
        $y(0)$ & $(0.2,0)$\\\hline
    \end{tabular}}
    \caption{\correz{Parameters used for Case study 1, see Figure \ref{fig_ex0b_p}. These values were obtained by sampling, respectively: $\beta_{ij} \sim U([0,1.2]), \delta_i = U([0,0.5]), \zeta_{ij} \sim U([0,1]), \xi_{ij} \sim U([0,1])$ for $i,j=1,2$, where $U$ denotes the uniform distribution.}}
    \label{tab:param1}
\end{table}

We consider system \eqref{eq_animfa} on a small network of $n=2$ communities, and consider an asymmetry in the functional responses between the two communities. We assume identical internal policies:
$$
\fcrsame(y_i,\bar{y})=\mathbbm{1}_{\{y_i\leq 0.2\}}, \quad \fbrsame(y_i,\bar{y})=\mathbbm{1}_{\{y_i> 0.2\}}, \quad i=1,2, 
$$
meaning a strict lockdown is imposed within each community as soon as the internal prevalence exceeds a threshold value of $0.2$, regardless of the average global prevalence. Oppositely, we assume that connections between the two communities are governed by asymmetrical rules:
$$
f_{\text{cr},12}(y_1,y_2,\bar{y})=\mathbbm{1}_{\{y_2\leq 0.1\}}, \quad f_{\text{br},12}(y_1,y_2,\bar{y})=\mathbbm{1}_{\{y_2> 0.1\}},$$
$$
f_{\text{cr},21}(y_1,y_2,\bar{y})=1-y_1y_2, \quad f_{\text{br},21}(y_1,y_2,\bar{y})=y_1y_2.
$$
This choice results, for values of the parameters corresponding to $R_0\approx 1.33$, in local and average prevalence, as well as all the link densities, quickly approaching stable limit cycles, as illustrated in Figure \ref{fig_ex0b_p}.

This asymmetry can be interpreted as community 1 being the ``hub'' in this network of communities, and community 2 being a peripheral node. The same behaviour can be observed with bigger communities, for example a star network of $n\geq 3$ communities with asymmetric responses between hub and peripheral nodes and vice versa. However, the simplest case of $n=2$ presented in Figure~\ref{fig_ex0b_p}, which corresponds to a system of only six ODEs, perfectly exemplifies the potential of our modelling approach in producing complex behaviour starting from a simple SIS epidemic model.

Even though the provided functional responses contain indicator functions, which are not $\mathcal{C}^1$-functions over the interval $[0,1]$, we emphasise that other choices of $f_{\text{cr,ii}}$, $f_{\text{br,ii}}$, $f_{\text{cr,ij}}$ and $f_{\text{br,ij}}$ can also lead to periodic behaviour. The benefit of having $\mathcal{C}^1$-functions is that the solution is guaranteed to be unique. One method is to approximate the indicator function, for which we verified that periodic solutions also occur. \correz{We refer, for a few options of smooth approximations of the indicator function, to Chapter 5.3 in \cite{evans2022partial}.}

Lastly, we remark that the existence of limit cycles is completely independent of the dependence of awareness on the \emph{global} prevalence $\bar{y}$. We investigate the influence of the dependence on the global prevalence in Section~\ref{cs5}.

\subsection{Case study \correz{2}: Internal and external connections}\label{C3}
To mitigate epidemics, policy-makers have several tools available. In the early phases of an epidemic, in which vaccines and drugs are still unavailable, only Non-Pharmaceutical Interventions (NPIs) can be utilised. Link-breaking policies are an important part of these NPIs. It is, however, not immediately clear whether it is preferable to remove connections within or between communities. By including a parameter in the system which weighs the relative magnitude of the corresponding response functions, we model a range of possible scenarios, ranging from total internal awareness and adaptivity (i.e., containment measures are taken \emph{within} each community) to completely total external awareness (i.e., containment measures are taken \emph{between} each couple of connected communities). We do so by considering the following choice for our functions:

\begin{align}\label{eq_cs3_fbr}
    \begin{split}
        \fbrsame &= c y_i^2, \\
        \fbrhet &= \frac{1}{c} y_i y_j.
    \end{split}
\end{align}
The constant $c$ allows for the balancing of the importance of internal and external responses to the epidemic. For $c=1$, the link-breaking functional response for external and internal connections is equivalent. Otherwise, for $c > 1$, links are broken at a higher rate \emph{within} communities and for $c<1$, links are broken at a higher rate \emph{between} communities. \sara{The remainder of the parameters and the structure of the complete graph can be found in the codes provided in  Section \ref{sec:codes}}. Figure~\ref{fig_cs3} depicts the time-evolving prevalence and link densities for three different situations. As expected, high link-breaking rates within communities results in less links within communities, and vice versa.

\begin{figure*}[!ht]
    \centering
    \subfloat[\label{fig_cs3a} Mainly break within communities]{%
        \includegraphics[width=0.33\textwidth]{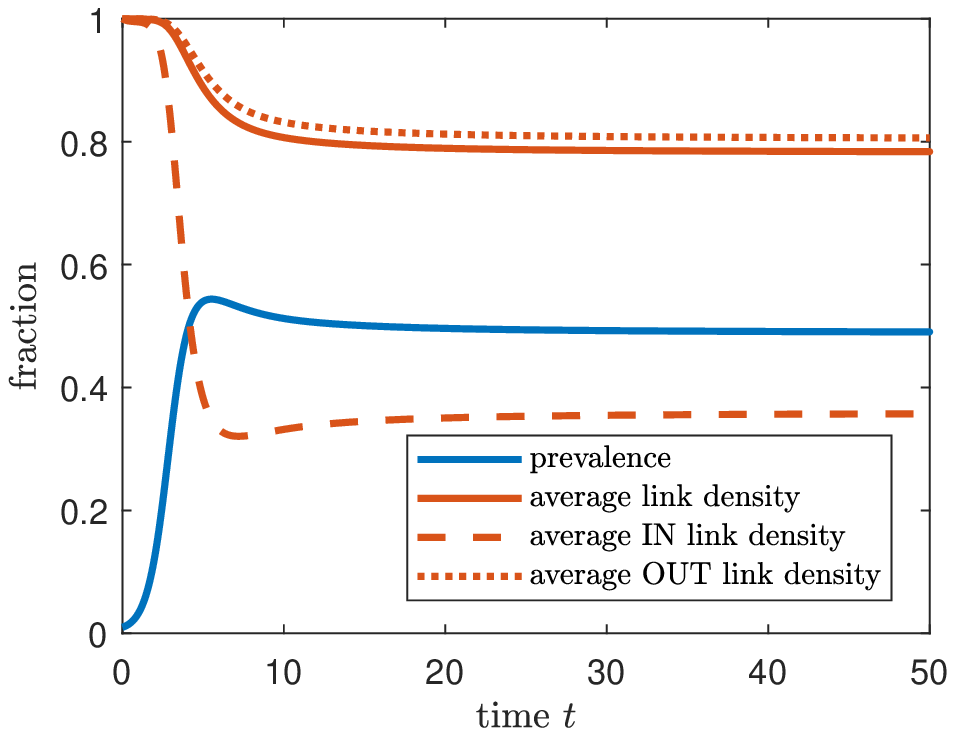}
    }
    \subfloat[\label{fig_cs3b} Balanced situation]{%
        \includegraphics[width=0.33\textwidth]{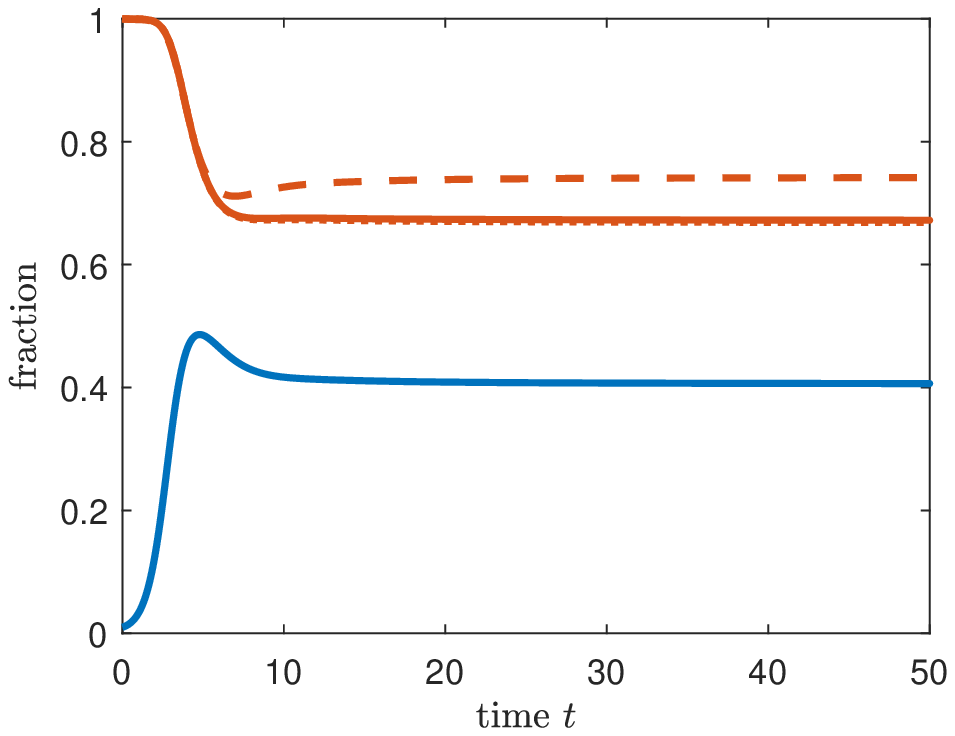}
    }
        \subfloat[\label{fig_cs3c} Mainly break between communities]{%
        \includegraphics[width=0.33\textwidth]{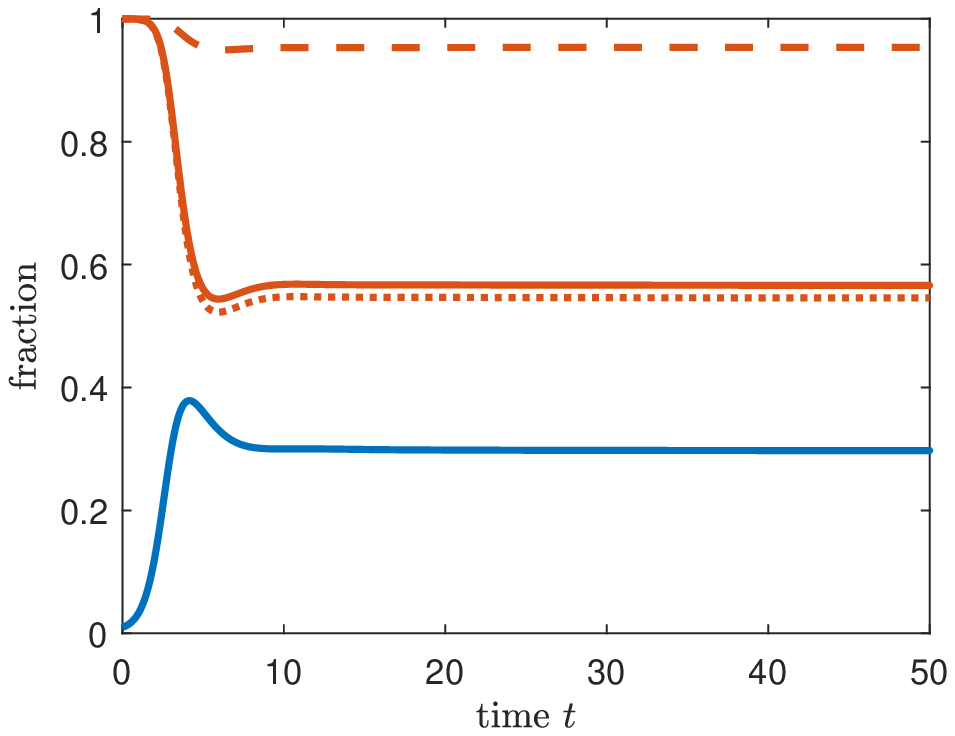}
    }
    \caption{Case study \correz{2}: influence of breaking within or between communities on the prevalence and link densities. \correz{For each figure we use $n=20$ and the values of parameters were obtained by sampling,
respectively: $\beta_{ij} \sim U([0,0.25]), \delta_i = 1, \zeta_{ij} \sim U([0,2]), \xi_{ij} \sim U([0,1])$ for all $i,j$, where $U$ denotes the uniform distribution}, resulting in $R_0 = 2.53$, with the exception that we used for (a), $c=4$, for (b), $c=1$ and for (c), $c=0.25$.}
    \label{fig_cs3}
\end{figure*}

Figure~\ref{fig_cs3} also suggests that breaking connections between communities may be preferable compared to breaking connections within communities to reduce the epidemic outbreak. We quantify the effectiveness of all methods using the peak prevalence~$y_p$ and the steady-state prevalence~$\yinfty$.

Figure~\ref{fig_cs3_complete_optimal} confirms for a wide range of $c$-values that breaking connections between communities is more successful in reducing the epidemic outbreak size than breaking connections within communities. We emphasise that this conclusion is highly dependent on the choice of the functional responses as well as the topology; different choices may result in significantly different behaviour.

\begin{figure*}[!ht]
    \centering
    \subfloat[\label{fig_cs3_complete_optimal} complete graph]{
    \includegraphics[width=0.49\textwidth]{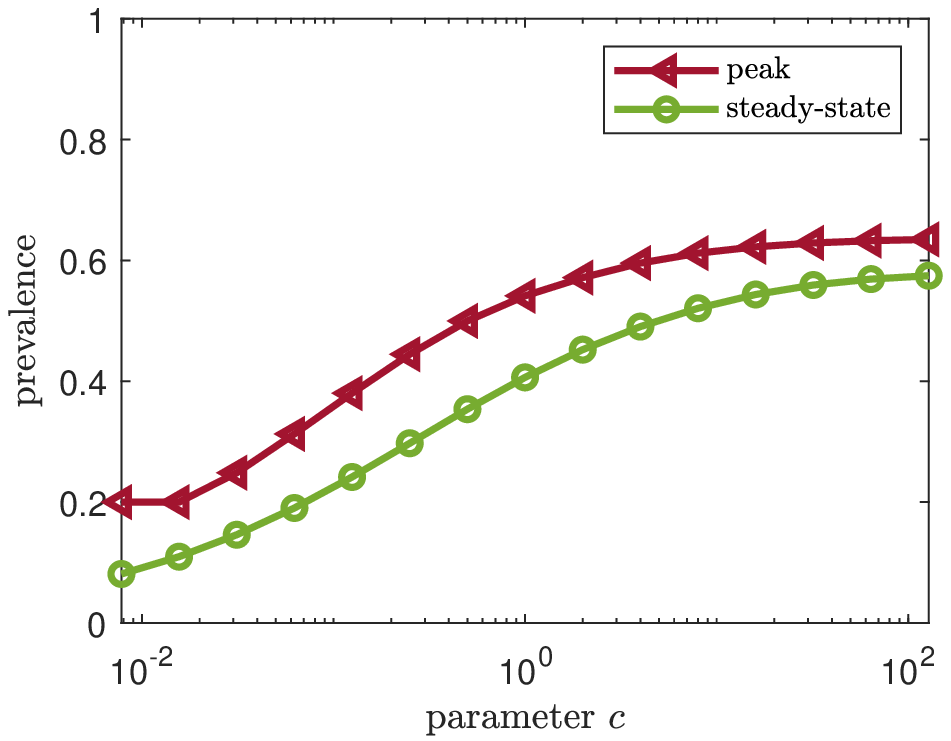}
    }
    \subfloat[\label{fig_cs3_cycle_optimal} cycle graph]{
    \includegraphics[width=0.49\textwidth]{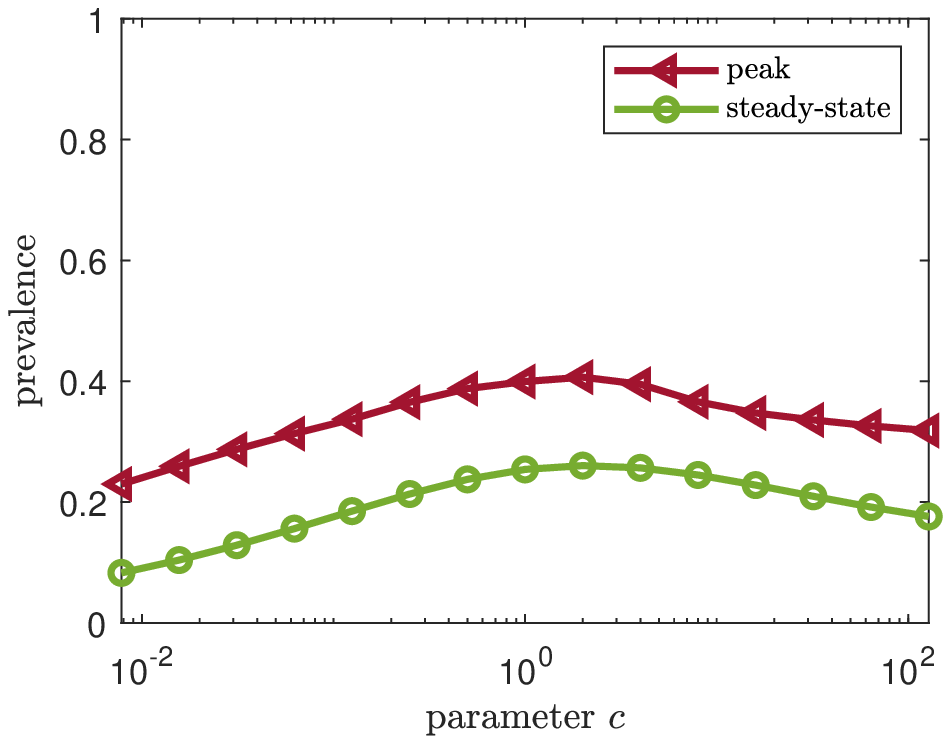}
    }
    \caption{Case study \correz{2}: impact of the parameter $c$ on the steady-state prevalence $\yinfty$ and the peak prevalence $y_p$ in the (a) complete graph and (b) cycle graph. Parameters are the same as Figure~\ref{fig_cs3}, except for (b), we have chosen \correz{a different sampling for $\beta_{ij} \sim U([0,1])$, where $U$ denotes the uniform distribution}, such that $R_0 = 1.83$. The horizontal axis is shown on a logarithmic scale.}
    \label{fig_cs3_optimal}
\end{figure*}

In particular, when considering a cycle graph, where each community is only connected to its two nearest neighbours, the situation changes drastically as observed in Figure~\ref{fig_cs3_cycle_optimal}. Compared to the balanced situation $c=1$ where links are broken equally fast within and between communities, both extremes $c \to 0$ and $c \to \infty$ appear to be beneficial for reducing the spread. Most likely, the spread of the disease can be diminished by one of two methods: (i) isolating communities with many infections or (ii) preventing getting infected by breaking connections between communities. This contrasts our results for the all-to-all topology in Figure~\ref{fig_cs3_complete_optimal}, where the method of quarantining communities (i.e.\ breaking links within communities) does not work very well, because of the huge number of neighbouring communities. Overall, we conclude that a cycle network is more localised and quarantining is therefore more effective than on the complete network.

\subsection{Case study \correz{3}: Local versus global awareness}\label{cs5}

Besides the influence of external and internal functional responses and the periodicity, another key aspect of the aNIMFA model is the possible dependence of the functional responses on the local and global prevalence, i.e.\ the choice for breaking and creating links can be dependent on the local \emph{and} global information on the disease. 

As in case study 2, the link-creation functional responses are taken as
\begin{align}\label{eq_cs4_fcr}
 \begin{split}
    \fcrsame &= 1-y_i^2, \\
    \fcrhet &= 1-y_i y_j,
    \end{split}
\end{align}
and for the link-breaking mechanism, we pick
\begin{align}\label{eq_cs4_fbr}
    \begin{split}
        \fbrsame &= y_i^2, \\
        \fbrhet &= c \ y_i y_j + (1-c) \ \bar{y}^2.
    \end{split}
\end{align}
Here, the constant $c$ balances between the breaking of links based on completely local information ($c=1$) and completely global information ($c=0$).

We emphasise that considering the global prevalence to make decisions whether to break connections is fundamentally different to taking all-to-all couplings. In a complete graph, the prevalence of all nodes influence a node's prevalence, whereas global awareness uses the average of all node prevalences to influence the link density directly, but cannot influence the node's prevalence directly.

For many graphs, graph sizes, functional responses and homogeneous and heterogeneous parameters, the peak prevalence~$y_p$ and the steady-state prevalence~$\yinfty$ are the highest with only global awareness ($c=0$) and decrease monotonically with increasing $c$. This is exemplified in Figure~\ref{fig_cs4f} in a \BA{} graph \cite{BarabasiBook}. The reasoning is as follows. Even though using global information may prevent spreading in some links on the network, it is only an average, necessarily also resulting in less link removals between other communities. Especially the most vulnerable parts of the network, i.e.\ those nodes with few infections and high neighbouring infections, have the most benefit by considering their neighbours directly instead of information on the network as a whole. In total, we can only conclude that local information is superior to suppress an epidemic for links between communities.

In another closely related scenario, we also balance between local and global awareness, but now for the \emph{internal} link densities $z_{ii}$, which changes the link-breaking mechanisms to
\begin{align}\label{eq_cs4_fbr_v2}
    \begin{split}
        \fbrsame &= c \ y_i^2 + (1-c) \ \bar{y}^2, \\
        \fbrhet &= y_i y_j.
    \end{split}
\end{align}
When a community with low prevalence is connected to a community with high prevalence, precautiously removing links within its own community may help to reduce the spread of the disease. Figure~\ref{fig_cs4e} supports our hypothesis by demonstrating that the steady-state prevalence~$\yinfty$ may exhibit non-monotonic behaviour in the parameter $c$. The effect of the parameter $c$ is admittedly small, but we believe the effect can be more substantial in larger networks. To conclude, for link removals within communities, there is no superior strategy -- depending on the network structure, the severity of the disease and the network dynamics, using global information may help to reduce the spreading of the disease. We expect that global information will become more important in larger networks, which we leave as a direction for further research.

\begin{figure}[H]
    \centering
    \subfloat[\label{fig_cs4e} Within communities]{%
    \includegraphics[width=0.49\textwidth]{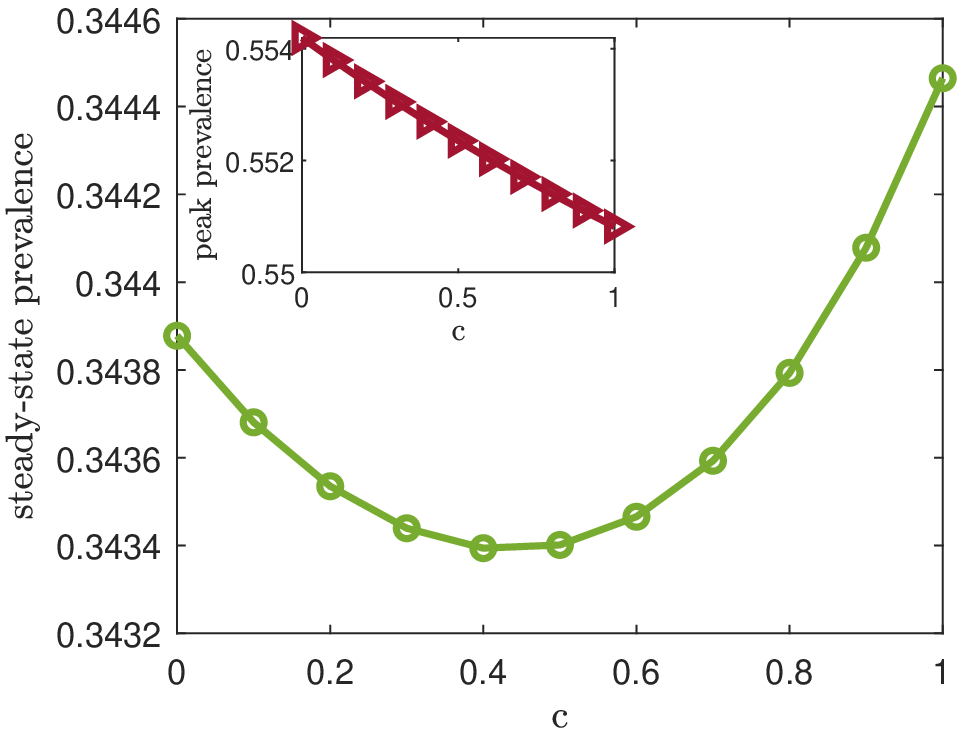}
    }
    \subfloat[\label{fig_cs4f} Between communities]{%
    \includegraphics[width=0.49\textwidth]{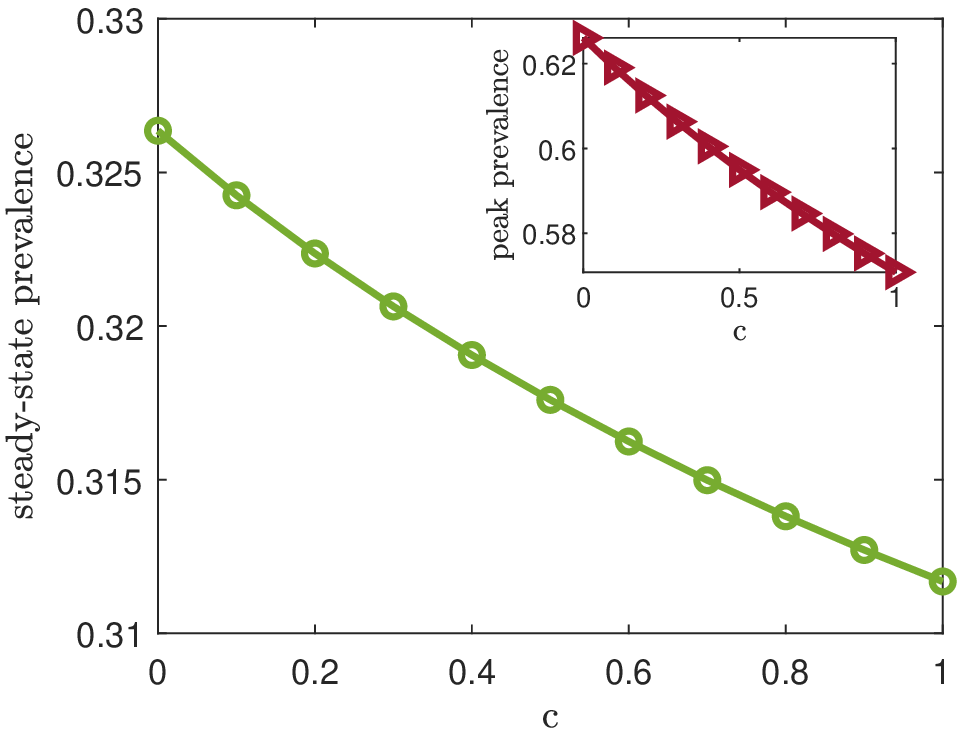}
    }
    \caption{Case study \correz{3}: balancing between local and global awareness of the disease (a) within communities and (b) between communities showing the peak prevalence (red triangles) and the steady-state prevalence (green circles). All curves are declining for increasing $c$ values, but the steady-state prevalence within communities is non-monotonic. Simulations are based on an \BA{} graph with $m_0=3, m=2$, \correz{with the values of parameters sampling: $\delta_i = 1, \zeta_{ij} \sim U([0, 2]), \xi_{ij} \sim U([0, 1])$, $y_i(0)=0$ for all $i \neq 1$ and $y_1(0)=0.2$ and $z_{ij}(0)=0$ for all $i,j$, where $U$ denotes the uniform distribution.} Furthermore, (a) $n=20$ nodes, $\beta_{ij} \sim U([0,1])$ such that $R_0 = 2.93$ and (b) $n=50$ nodes, $\beta_{ij} \sim U([0, 0.8])$ such that $R_0 = 2.60$.} 
    \label{fig_cs4}
\end{figure}

\section{Conclusions}\label{sec:concl}
In this paper, we proposed a generalization of the aNIMFA model investigated in \cite{achterbergsensi2022adaptive} to a network of $n$~communities. We provided a formula for the well-known basic reproduction number~$R_0$. We investigated existence and stability of equilibria of the system. As usual with this type of models, the stability of the Disease Free Equilibrium depends the threshold value $R_0$ being smaller or greater than $1$. We also showed that at least one Endemic Equilibrium exists under a condition slightly stronger than $R_0>1$. 

Lastly, we enhanced our analytical results through an extensive numerical exploration of various case studies. In particular, we showed that the multigroup aNIMFA model allows for periodic orbits in a network with just two communities. We showcased this fact with an asymmetric on/off strategy for the adaptivity between communities, representing total lockdowns when the prevalence reaches a given threshold. 
Then, we showed that breaking connections between communities is preferable over breaking connections within communities in dense graphs, but in sparse networks (with few links) both link-breaking strategies are viable approaches to lessen the disease prevalence. Lastly, we measured the trade-off between local and global awareness on the current state of the epidemic. In many scenarios, using local information appears superior to using global awareness, but further research on larger networks is required to fully support this claim.

Given that many of our conclusions are drawn based on numerical solutions, it would be interesting to derive an analytical foundation of these results, while keeping restrictions on the functional responses $\fbrhet$ and $\fcrhet$ as minimal as possible. Perhaps monotonicity of the functional responses is sufficient to deduce information on the asymptotic behaviour of the system. Or is there any other characteristic one could leverage to foresee relevant characteristics of the system, e.g. the peak of prevalence or the steady-state prevalence?

Moreover, we remark that our approach to the modelling of adaptivity in the context of epidemic spreading can be easily generalized to more complex models. The main underlying hypothesis of the systems presented here and in \cite{achterbergsensi2022adaptive} is the SIS compartmental structure, which does not allow for any kind of immunity. However, modelling the connectivity based on link densities between between nodes and viral states can be included in any multigroup model, as a level of adaptation of the communities to the evolution of an epidemic. In more complex models, such as the SIR (Susceptible -- Infected -- Recovered), SIRS, SAIRS (Susceptible -- Asymptomatic -- symptomatic Infectious -- Recovered -- Susceptible) or SIRWS (Susceptible -- Infectious -- Recovered -- Waning immunity -- Susceptible), incorporating disease awareness and adaptivity in the system is crucial for improving its realism. Our approach to the subject is very simple, and therefore well-suitable for wide applicability in epidemics. \correz{In particular, one could consider a variation of system \eqref{eq_animfa} in which some of the parameters, e.g. $\beta_{ij}$, vary periodic in time, mimicking seasonality.}

Lastly, one may consider the inclusion of time-delayed mechanisms and multiple viruses spreading in the same population. For the former, it could be interesting to investigate an adaptivity based on e.g. yesterday's values on the prevalence, which are more realistically available than real-time ones. For the latter, one could have different adaptivity strategies depending on how each virus is perceived by the population, greatly increasing the complexity of the dynamics.

\correz{\section{Data Availability Statement}\label{sec:codes}
All the codes used to generate Figures \ref{fig_ex0b_p}-\ref{fig_cs4} are avaiable in the GitHub repository \newline\url{https://github.com/SaraSottile/AdaptativeMeanField}.
}

\vspace{1cm}

\noindent \textbf{Acknowledgements.} Mattia Sensi was supported by the Italian Ministry for University and Research (MUR) through the PRIN 2020 project ``Integrated Mathematical Approaches to Socio-Epidemiological Dynamics'' (No. 2020JLWP23, CUP: E15F21005420006).

{\footnotesize
    \bibliographystyle{unsrt}
    \bibliography{references}

\begin{thebibliography}{10}

\bibitem{achterbergsensi2022adaptive}
M.~A. Achterberg and M.~Sensi.
\newblock A minimal model for adaptive {SIS} epidemics.
\newblock {\em Nonlinear Dynamics}, 111:12657--12670, 2023.

\bibitem{juher2020robustness}
D.~Juher, D.~Rojas, and J.~Salda{\~n}a.
\newblock Robustness of behaviorally induced oscillations in epidemic models
  under a low rate of imported cases.
\newblock {\em Physical Review E}, 102(5):052301, 2020.

\bibitem{juher2022saddle}
D.~Juher, D.~Rojas, and J.~Salda{\~n}a.
\newblock Saddle--node bifurcation of limit cycles in an epidemic model with
  two levels of awareness.
\newblock {\em Physica D: Nonlinear Phenomena}, 448:133714, 2023.

\bibitem{just2018oscillations}
W.~Just, J.~Salda{\~n}a, and Y.~Xin.
\newblock Oscillations in epidemic models with spread of awareness.
\newblock {\em Journal of Mathematical Biology}, 76(4):1027--1057, 2018.

\bibitem{sahneh2012optimal}
F.~D. Sahneh and C.~M. Scoglio.
\newblock Optimal information dissemination in epidemic networks.
\newblock In {\em 2012 IEEE 51st IEEE conference on decision and control
  (cdc)}, pages 1657--1662. IEEE, 2012.

\bibitem{sahneh_2019_awareness_SIS}
F.~D. Sahneh, A.~Vajdi, J.~Melander, and C.~M. Scoglio.
\newblock {Contact Adaption During Epidemics: A Multilayer Network Formulation
  Approach}.
\newblock {\em IEEE Transactions on Network Science and Engineering},
  6(1):16--30, 2019.

\bibitem{agliari2006efficiency}
E.~Agliari, R.~Burioni, D.~Cassi, and F.~M. Neri.
\newblock Efficiency of information spreading in a population of diffusing
  agents.
\newblock {\em Physical Review E}, 73(4):046138, 2006.

\bibitem{funk2009spread}
S.~Funk, E.~Gilad, C.~Watkins, and V.~A.~A. Jansen.
\newblock The spread of awareness and its impact on epidemic outbreaks.
\newblock {\em Proceedings of the National Academy of Sciences},
  106(16):6872--6877, 2009.

\bibitem{ASISpairapprox}
A.~Szabó-Solticzky, L.~Berthouze, I.~Z. Kiss, and P.~L. Simon.
\newblock {Oscillating epidemics in a dynamic network model: stochastic and
  mean-field analysis}.
\newblock {\em J. Math. Bio.}, 72:1153--1176, 2016.

\bibitem{kouzy2020covid}
R.~Kouzy, J.~Abi~Jaoude, A.~Kraitem, M.~B. El~Alam, B.~Karam, E.~Adib,
  J.~Zarka, C.~Traboulsi, E.~W. Akl, and K.~Baddour.
\newblock {{C}oronavirus {G}oes {V}iral: {Q}uantifying the {COVID}-19
  {M}isinformation {E}pidemic on {T}witter}.
\newblock {\em Cureus}, 12, 2020.

\bibitem{GrossEtAl2006}
T.~Gross, C.~J.~D. D'Lima, and B.~Blasius.
\newblock Epidemic dynamics on an adaptive network.
\newblock {\em Phys. Rev. Lett.}, 96:208701, May 2006.

\bibitem{RandomLinkDeActivate}
I.~Z. Kiss, L.~Berthouze, T.~J. Taylor, and P.~L. Simon.
\newblock {Modelling approaches for simple dynamic networks and applications to
  disease transmission models}.
\newblock {\em Proc. R. Soc. A}, 468:1332--1355, 2012.

\bibitem{G-ASIS}
M.~A. Achterberg, J.~L.~A. Dubbeldam, C.~J. Stam, and P.~Van~Mieghem.
\newblock Classification of link-breaking and link-creation updating rules in
  susceptible-infected-susceptible epidemics on adaptive networks.
\newblock {\em Phys. Rev. E}, 101:052302, May 2020.

\bibitem{cascante2022adaptive}
J.~Cascante-Vega, S.~Torres-Florez, J.~Cordovez, and M.~Santos-Vega.
\newblock How disease risk awareness modulates transmission: coupling
  infectious disease models with behavioural dynamics.
\newblock {\em Royal Society Open Science}, 9(1):210803, 2022.

\bibitem{LajmanovichYorke1976}
A.~Lajmanovich and J.~A. Yorke.
\newblock A deterministic model for gonorrhea in a nonhomogeneous population.
\newblock {\em Mathematical Biosciences}, 28(3):221--236, 1976.

\bibitem{huang1992stability}
W.~Huang, K.~L. Cooke, and C.~Castillo-Chavez.
\newblock Stability and bifurcation for a multiple-group model for the dynamics
  of {HIV/AIDS} transmission.
\newblock {\em SIAM Journal on Applied Mathematics}, 52(3):835--854, 1992.

\bibitem{sun2011global}
R.~Sun and J.~Shi.
\newblock Global stability of multigroup epidemic model with group mixing and
  nonlinear incidence rates.
\newblock {\em Applied Mathematics and Computation}, 218(2):280--286, 2011.

\bibitem{muroya2013global}
Y.~Muroya, Y.~Enatsu, and T.~Kuniya.
\newblock Global stability for a multi-group {SIRS} epidemic model with varying
  population sizes.
\newblock {\em Nonlinear Analysis: Real World Applications}, 14(3):1693--1704,
  2013.

\bibitem{muroya2014further}
Y.~Muroya and T.~Kuniya.
\newblock Further stability analysis for a multi-group {SIRS} epidemic model
  with varying total population size.
\newblock {\em Applied Mathematics Letters}, 38:73--78, 2014.

\bibitem{mohapatra2015compartmental}
R.~N. Mohapatra, D.~Porchia, and Z.~Shuai.
\newblock Compartmental disease models with heterogeneous populations: a
  survey.
\newblock In {\em Mathematical Analysis and its Applications: Roorkee, India,
  December 2014}, pages 619--631. Springer, 2015.

\bibitem{fan2018global}
D.~Fan, P.~Hao, D.~Sun, and J.~Wei.
\newblock Global stability of multi-group {SEIRS} epidemic models with
  vaccination.
\newblock {\em International Journal of Biomathematics}, 11(01):1850006, 2018.

\bibitem{adimy2023multigroup}
M.~Adimy, A.~Chekroun, L.~Pujo-Menjouet, and M.~Sensi.
\newblock A multigroup approach to delayed prion production.
\newblock {\em Discrete and Continuous Dynamical Systems-B}, 29(7):2972--2998,
  2024.

\bibitem{ottaviano2023global}
S.~Ottaviano, M.~Sensi, and S.~Sottile.
\newblock Global stability of multi-group {SAIRS} epidemic models.
\newblock {\em Mathematical Methods in the Applied Sciences},
  46(13):14045--14071, 2023.

\bibitem{boccalini2021health}
S.~Boccalini, E.~Pariani, G.~E. Calabr{\`o}, C.~De~Waure, D.~Panatto,
  D.~Amicizia, P.~L. Lai, C.~Rizzo, E.~Amodio, F.~Vitale, et~al.
\newblock Health technology assessment ({HTA}) of the introduction of influenza
  vaccination for {I}talian children with {F}luenz {T}etra{\textregistered}.
\newblock {\em Journal of Preventive Medicine and Hygiene}, 62(2):E1--E118,
  2021.

\bibitem{calabro2022health}
G.~E. Calabr{\`o}, S.~Boccalini, A.~Bechini, D.~Panatto, A.~Domnich, P.~L. Lai,
  D.~Amicizia, C.~Rizzo, A.~Pugliese, M.~L. Di~Pietro, et~al.
\newblock Health {T}echnology {A}ssessment: a value-based tool for the
  evaluation of healthcare technologies. reassessment of the
  cell-culture-derived quadrivalent influenza vaccine: {F}lucelvax
  {T}etra{\textregistered} 2.0.
\newblock {\em Journal of preventive medicine and hygiene}, 63(4/Supplement
  1):E1--E140, 2022.

\bibitem{fochesato2022economic}
A.~Fochesato, S.~Sottile, A.~Pugliese, S.~M{\'a}rquez-Pel{\'a}ez, H.~Toro-Diaz,
  R.~Gani, P.~Alvarez, and J.~Ruiz-Arag{\'o}n.
\newblock An economic evaluation of the adjuvanted quadrivalent influenza
  vaccine compared with standard-dose quadrivalent influenza vaccine in the
  spanish older adult population.
\newblock {\em Vaccines}, 10(8):1360, 2022.

\bibitem{IntroNIMFA}
P.~Van~Mieghem.
\newblock {The N-intertwined SIS epidemic network model}.
\newblock {\em Computing}, 93:147--169, Dec 2011.

\bibitem{achterberg2020classification}
M.~A. Achterberg, J.~L.~A. Dubbeldam, C.~J. Stam, and P.~Van~Mieghem.
\newblock Classification of link-breaking and link-creation updating rules in
  susceptible-infected-susceptible epidemics on adaptive networks.
\newblock {\em Physical Review E}, 101(5):052302, 2020.

\bibitem{achterberg2022moment}
M.~A. Achterberg and P.~Van~Mieghem.
\newblock Moment closure approximations of susceptible-infected-susceptible
  epidemics on adaptive networks.
\newblock {\em Physical Review E}, 106(1):014308, 2022.

\bibitem{van2008further}
P.~Van~den Driessche and J.~Watmough.
\newblock Further notes on the basic reproduction number.
\newblock {\em Mathematical epidemiology}, pages 159--178, 2008.

\bibitem{kulpa1997poincare}
W.~Kulpa.
\newblock The {P}oincar{\'e}-{M}iranda theorem.
\newblock {\em The American Mathematical Monthly}, 104(6):545--550, 1997.

\bibitem{mawhin2019simple}
J.~Mawhin.
\newblock {Simple proofs of the Hadamard and Poincar{\'e}--Miranda theorems
  using the Brouwer fixed point theorem}.
\newblock {\em The American Mathematical Monthly}, 126(3):260--263, 2019.

\bibitem{evans2022partial}
L.~C. Evans.
\newblock {\em {Partial differential equations}}, volume~19.
\newblock {American Mathematical Society}, 2022.

\bibitem{BarabasiBook}
A.~Barabási.
\newblock {\em Network science}.
\newblock Cambridge University Press, Cambridge, United Kingdom, Jul 2016.

\end{thebibliography}
}

\end{document}